\documentclass[preprint,12pt]{elsarticle}
\usepackage{amssymb}
\usepackage{amsthm}
\journal{Elsevier}
\usepackage[dvipsnames]{xcolor}

\usepackage{amsmath}
\usepackage{amsfonts}
\usepackage{amssymb}
\usepackage{enumerate}
\usepackage{lineno}
\usepackage{tikz} \usetikzlibrary{calc} 

\newtheorem{definition}{Definition}[section]
\newtheorem{proposition}[definition]{Proposition}
\newtheorem{lemma}[definition]{Lemma}
\newtheorem{theorem}[definition]{Theorem}
\newtheorem{corollary}[definition]{Corollary}

\newdefinition{example}[definition]{Example}
\newdefinition{remark}[definition]{ Remark}
\newdefinition{problem}[definition]{ Problem}
\newdefinition{question}[definition]{Question}
\newdefinition{fact}[definition]{Fact}
\newproof{pot}{Proof}

\begin{document}

\begin{frontmatter}
\title{About the uniqueness of the hyperspaces $C(p,X)$ in some classes of continua}

\author{Florencio~Corona--V\'azquez}
\ead{florencio.corona@unach.mx}

\author{Russell~Aar\'on~Qui\~nones--Estrella \corref{cor1}}
\ead{rusell.quinones@unach.mx}

\author{Javier~S\'anchez--Mart\'inez}
\ead{jsanchezm@unach.mx}

\cortext[cor1]{Corresponding author}

\address{Universidad Aut\'onoma de Chiapas, Facultad de Ciencias en F\'isica y Matem\'aticas, Carretera Emiliano Zapata Km. 8, Rancho San Francisco, Ter\'an, C.P. 29050, Tuxtla Guti\'errez, Chiapas, Mexico.}

\begin{abstract}
Given a continuum $X$ and $p\in X$, we will consider the hyperspace $C(p,X)$ of all subcontinua of $X$ containing $p$. Given a family of continua $\mathcal{C}$, a continuum $X\in\mathcal{C}$ and $p\in X$, we say that $(X,p)$ has unique hyperspace $C(p,X)$ relative to $\mathcal{C}$ if for each $Y\in\mathcal{C}$ and $q\in Y$ such that $C(p,X)$ and $C(q,Y)$ are homeomorphic, then there is an homeomorphism between $X$ and $Y$ sending $p$ to $q$. In this paper we show that $(X,p)$ has unique hyperspace $C(p,X)$ relative to the classes of dendrites if and only if $X$ is a tree, we present also some classes of continua without unique hyperspace $C(p,X)$; this answer some questions posed in \cite{Corona.et.al(2019)}. 
\end{abstract}


\end{frontmatter}
\section{Introduction}
A \textit{continuum} is a nonempty compact connected metric space. Given a continuum $X$, by a \textit{hyperspace} of $X$ we mean a specified collection of subsets of $X$. In the literature, some of the most studied hyperspaces are the following:   
\begin{align*}
2^{X} & = \{ A \subset X : A \text{ is nonempty and closed}\},        \\
C(X) & = \{ A \subset X : A \text{ is nonempty, connected and closed}\}, \\
C(P,X) & = \{A\in C(X) : P\subset A\} \textrm{, where $P\in 2^X$}.
\end{align*}

The collection $2^{X}$ is called the \textit{hyperspace of closed subsets} of $X$ whereas that $C(X)$ is called the \textit{hyperspace of subcontinua} of $X$. These hyperspaces are endowed with the Hausdorff metric, see \cite[p. 1]{Nadler(1978)}. The main objects of interest in this paper are the hyperspaces $C(p,X):=C(\{p\},X)$, where $p\in X$. 
 

In general, given a hyperspace $\mathcal{H}(X)\in \{2^{X},C(X)\}$, $X$ is said to \textit{have unique hyperspace $\mathcal{H}(X)$} if for each continuum $Y$ such that $\mathcal{H}(X)$ is homeomorphic to $\mathcal{H}(Y)$, it holds that $X$ is homeomorphic to $Y$. In a similar setting, the following concept was defined in \cite{Corona.et.al(2019)}: Given a class of continua $\mathcal{C}$, a continuum $X\in\mathcal{C}$ and $p\in X$, $(X,p)$ is said to \textit{have unique hyperspace $C(p,X)$ in (or relative to) $\mathcal{C}$} if for each $Y\in\mathcal{C}$ and $q\in Y$ such that $C(p,X)$ is homeomorphic to $C(q,Y)$, there is an homeomorphism $h:X\to Y$ such that $h(p)=q$.

By a \textit{finite graph} we mean a continuum $X$ which can be written as the union of finitely many arcs, any two of which are either disjoint or intersect only in one or both of their end points. A \textit{tree} is a finite graph without simple closed curves. A \textit{dendrite} is a locally connected continuum without simple closed curves. 

The main result in  \cite{Corona.et.al(2019)} is the following:

\begin{theorem}\cite[Theorem~4.14]{Corona.et.al(2019)}\label{trees}
Let $X$ be a tree and $p\in X$. Then $(X,p)$ has unique hyperspace $C(p,X)$ in the class of trees.
\end{theorem}

In the same paper the authors posed the following questions:

\begin{question}\label{dendrites}
If $X$ is a dendrite and $p\in X$, has $(X,p)$ unique hyperspace $C(p,X)$ in the class of dendrites?
\end{question}

\begin{question}\label{continua}
Are there a continuum $X$ and $p\in X$ such that $(X,p)$ has unique hyperspace in the class of continua?
\end{question}

In this paper we give a negative answer to the Question \ref{dendrites}, nevertheless we show that, in the class of dendrites, the trees are the only continua having unique hyperspace $C(p,X)$ (see Theorem \ref{mainresult} below). We present also partial solutions to Question \ref{continua} by showing some classes of continua without unique hyperspace $C(p,X)$.

\section{Definitions and preliminaries}

In this paper, \textit{dimension} means inductive dimension as defined in \cite[(0.44), p.~21]{Nadler(1978)}, following the author of that book, we will denote by $\dim(X)$ and $\dim_{p}(X)$ the dimension of the space $X$ and the dimension of $X$ at $p$, respectively. For a subset $A$ of $X$ we denote by $|A|$ the cardinality of $A$
and we use, as customary, the symbols $\textrm{Int}_{X}(A)$, $Cl_{X}(A)$ and $Fr_{X}(A)$ to denote the interior, closure and boundary of $A$ in $X$, respectively. If there is no confusion, we will write simply $\textrm{Int}(A)$, $Cl(A)$ and $Fr(A)$. We will use the symbol $\mathbb{N}$ to denote the set of all positive integers.

Given a continuum $X$, $p\in X$ and $\beta$ be a cardinal number, we say that $p$ \textit{is of order less than or equal to $\beta$ in $X$}, written $\textrm{ord}(p,X)\leq \beta$, provided that for each open subset $U$ of $X$ containing $p$, there exists an open subset $V$ of $X$ such that $p\in V\subset U$ and $|Fr(V)|\leq \beta$. We say that $p$ \textit{is the order $\beta$ in $X$}, written $\textrm{ord}(p,X)=\beta$ provided that $\textrm{ord}(p,X)\leq \beta$ and $\textrm{ord}(p,X)\nleq \alpha$ for any cardinal number $\alpha<\beta$. Since in this paper we are interested only in the finite and countable cases, we will write $\textrm{ord}(p,X)<\infty$ if $\textrm{ord}(p,X)$ is an integer and $\text{ord}(p,X)=\infty$ in any other case (no matter wich infinite cardinal number is about). 

It is well known that, if $Z$ is a finite graph and $z\in Z$, then $\textrm{ord}(z,Z)< \infty$ (see \cite[Theorem~9.10, p.~144]{Nadler(1992)}) and if $Y$ is a dendrite and $y\in Y$, then $\textrm{ord}(y,Y)\leq |\mathbb{N}|=\aleph_{0}$ (see \cite[Corollary~10.20.1, p.~173]{Nadler(1992)}).

As customary, for a topological space $Z$, $\pi _{0}(Z)$ denote the set of connected components of $Z$, also, if $Z$ is connected, it is defined the \textit{component number of $z$ in $Z$}, written $c(z,Z)$, as the cardinality of $\pi_{0}(Z-\{z\})$ (see \cite[Definition 10.11]{Nadler(1992)}). Whit this notation, by \cite[Theorem 10.13, 170]{Nadler(1992)}, if $X$ is a dendrite and $x\in X$ then $c(x,X)=\textrm{ord}(x,X)$. In the case that $\textrm{ord}(x,X)=1$ we say that $x$ is an \textit{end point of $X$}, if $\textrm{ord}(x,X)\geq 3$ we say that $x$ is a \textit{ramification point of $X$}. We will denote by $E(X)$ and $R(X)$, respectively, the set of all end points and the set of all ramification points in $X$.

An \textit{arc} is any space which is homeomorphic to the closed interval $[0,1]$. A \textit{dendroid} is an arcwise continuum $X$, such that for each $A,B\in C(X)$ holds $A\cap B\in C(X)$. In this paper, a dendroid $X$ is say to be \textit{smooth at $p\in X$} if satisfies the definition given in \cite[p. 298]{CharatonikEberhart(1970)}.  

We denote by $\mathcal{T}$ and $\mathcal{D}$ the classes of all trees and all dendrites, respectively. By \cite[Proposition~9.4, p.~142]{Nadler(1992)}, each finite graph is locally connected and then each tree is a dendrite, thus $\mathcal{T}\subset \mathcal{D}$. 

We denote by $\mathcal{Q}$ to the Hilbert Cube, i.e. $\mathcal{Q}$ is the product of infinite countable many copies of the unit interval $[0,1]$ with the product topology.

\section{Dendrites with unique hyperspace $C(p,X)$}

We begin this section with a characterization of trees in the class of dendrites using the structure of the hyperspaces $C(p,X)$.

\begin{theorem}\label{trees_dendrites}
Let $X$ be a dendrite. The following conditions are equivalent:
\begin{enumerate}
\item There exists $p\in X$ such that $\dim C(p,X)<\infty$.
\item $X$ satisfies the next two conditions:
\begin{itemize}
\item $\textrm{ord}(x,X)<\infty$ for all $x\in X$;
\item $\textrm{ord}(x,X)\leq 2$ for all but finitely many $x\in X$.
\end{itemize}
\item $X$ is a tree.
\item $\dim C(x,X)<\infty$ for each $x\in X$.
\end{enumerate}
\end{theorem}

\begin{proof}
To see $1.$ implies $2.$, suppose by the contrary that there exists $x\in X$ such that $\textrm{ord}(x,X)=\infty$. By \cite[Corollary 3.12, p. 1006]{Pellicer(2007)},  $x\neq p$. Since $X-\{x\}$ has countable many components $\{A_{i}\}_{i\in\mathbb{N}\cup \{0\}}$ (we can assume that $p\in A_{0}$) and for each $i\in\mathbb{N}$, $A_{i}\cup\{x\}\in C(X)$, we can consider a countable many order arcs $\alpha_{i}:[0,1]\to C(X)$ such that $\alpha_{i}(0)=\{x\}$ and $\alpha_{i}(1)=A_{i}\cup \{x\}$. Let $f:\mathcal{Q}\to C(p,X)$ given by $$f((t_{i})_{i\in\mathbb{N}})=A_{0}\cup \bigcup^{\infty}_{i=1}\alpha_{i}\left(\frac{t_{i}}{2^{i}}\right), \textrm{ for each $(t_{i})_{i\in\mathbb{N}}\in \mathcal{Q}$}.$$
It is easy to see that $f$ is an embedding and therefore $\dim C(p,X)=\infty$, which is a contradiction.\\
To see the second point, we will proceed again by contradiction. We suppose that there exists a sequence $\{x_{n}\}_{n\in\mathbb{N}}$ of different points in $X$ such that $\textrm{ord}(x_{i},X)\geq 3$ for each $n\in\mathbb{N}$. By \cite[Lemma 9.11, p. 145]{Nadler(1992)}, there exists a subcontinuum $K$ of $X$ such that $\textrm{ord}(K,X)=\infty$. Using the same idea of the first part, replacing $K$ by $\{x\}$ we have that $\dim C(p,X)=\infty$, which is impossible.

By \cite[Theorem 9.10, p. 144]{Nadler(1992)} we have that $2.$ and $3.$ are equivalent.

To see, $3.$ implies $4.$ suppose that $X$ is a tree. In this case, if $n=\sum_{q\in R(X)}\textrm{ord}(q,X)$, then using \cite[Theorem 3.11 and Corollary 3.12]{Pellicer(2007)} we have that $\dim C(x,X)\leq n$ for each $x\in X$.

Finally $4.$ implies trivially $1.$
\end{proof}

Now, we show an extension of Theorem \ref{trees}.

\begin{lemma}\label{uniquetrees}
Let $X\in\mathcal{T}$ and $p\in X$. Then $(X,p)$ has unique hyperspace $C(p,X)$ relative to $\mathcal{D}$. 
\end{lemma}

\begin{proof}
Let $Y\in\mathcal{D}$ and let $q\in Y$ such that $C(p,X)$ is homeomorphic to $C(q,Y)$. Since $\dim C(q,Y)=\dim C(p,X)<\infty$, by Proposition \ref{trees_dendrites}, we obtain that $Y\in \mathcal{T}$. From Theorem \ref{trees}, we conclude that there exists an homeomorphism between $X$ and $Y$ sending $p$ to $q$. 
\end{proof}

Similarly to Theorem \ref{trees_dendrites}, in the following result we present another characterization of the trees in the class of dendrites by using dimension of $C(X)$.

\begin{theorem}
Let $X$ be a dendrite. The following conditions are equivalent:
\begin{enumerate}
\item $\dim C(X)<\infty$,
\item $X$ is a tree,
\item $\dim C(p,X)<\infty$ for some $p\in X$,
\item $\dim C(p,X)<\infty$ for each $p\in X$.
\end{enumerate}
\end{theorem}
\begin{proof}
By Theorem \ref{trees_dendrites} we have that $2.$, $3.$ and $4.$ are equivalent. We will prove that $1.$ and $2.$ are equivalent. First, suppose that $\dim C(X)<\infty$, so for each $x\in X$, $\dim C(x,X)\leq \dim C(X)<\infty$ and then $X$ is a tree by Theorem \ref{trees_dendrites}. Now, if $X$ is a tree, then by \cite[Theorem 2.4, p. 791]{Veronica(2006)}, $\dim C(X)<\infty$. 
\end{proof}

Given a dendrite $X$, we say that $p\in X$ is a \textit{semi-hairy point in $X$} if $p$ satisfies one of the following two conditions:
\begin{itemize}
\item $\textrm{ord}(p,X)=\infty$, or
\item $p\in Cl(R(X)-\{p\})$.
\end{itemize}
We denote by $sh(X)$ the set of all semi-hairy points in $X$.

\begin{remark}
It follows from \cite[Theorem 9.10, p. 144]{Nadler(1992)} that if $X$ is a dendrite, then $X$ is a tree if and only if $sh(X)=\emptyset$.
\end{remark}

The next result is an easy application of \cite[Theorem 4, p. 221]{Eberhart(1978)}.

\begin{proposition}\label{interiorgraph}
Let $X$ be a dendrite. For a point $p\in X$ are equivalent:
\begin{enumerate}
\item $p\in sh(X)$,
\item $\dim_{\{p\}}(C(p,X))=\infty$,
\item $C(p,X)$ is homeomorphic to $\mathcal{Q}$,
\item $p$ is not in the interior (relative to $X$) of a finite graph in $X$.
\end{enumerate}
\end{proposition}

As a consequence of the last result, we are able to give an answer to Question \ref{dendrites}, the answer is negative as the following example shows. 

\begin{example}\label{examples}
If $F_{\omega}$ and $W$ are the dendrites known as \textit{hairy point } \cite[bottom p. 
46]{IllanesNadler(1999)} and \textit{null comb} \cite[Exercise 6.4 and Remark p. 50]{IllanesNadler(1999)}, then 
$C(v,F_{\omega})$ and $C(w,W)$ are homeomorphic because both are homeomorphic to $\mathcal{Q}$ \cite[part of 
Exercise 61 p. 48 and Exercise 6.4 p. 50]{IllanesNadler(1999)}, but $F_{\omega}$ and $W$ are not 
homeomorphic.
 
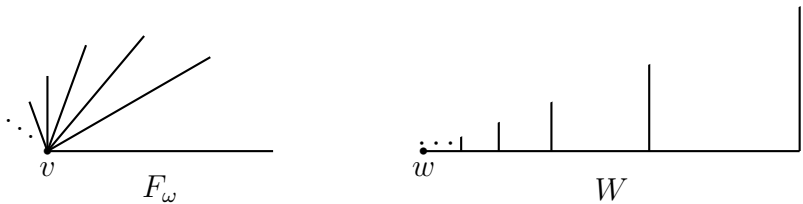
\begin{center}
 \begin{tikzpicture}
 \begin{scope}
 \draw node at (0,0){\tiny $\bullet$} 
	node at (0,0)[below]{$v$}
	node at (0,0)[above left]{\small$\ddots$};
\foreach \r/\a in{3/0,2.5/30,2/50,1.5/70,1/90, 0.7/110}{
  \draw[thick] (0,0) -- (\a:\r cm);
}
\draw node at (1.5, -0.5){$F_{\omega}$};
\end{scope}
\begin{scope}[xshift=5cm]
\draw[thick] (0,0) -- (5,0);
\draw node at (0,0)[below]{$w$}
      node at (0.2,0.1){$\cdots$}
      node at (0,0){\tiny $\bullet$}
      node at (2.5,-0.5){$W$};
\clip (0,0) -- (5.2,0) -- (5.2,2) -- cycle;
\foreach \x in{0.5,1,1.7,3,5}{
  \draw[thick] (\x,0) -- (\x,2);
}
\end{scope}

 \end{tikzpicture}
 \end{center}
\end{example}

As a nice consequence of the above proposition and example, we get the following more general result.

\begin{proposition}\label{peludos}
If $X$ is a dendrite and $p\in sh(X)$, then there exists a dendrite $Y$ and $p^{\prime}\in Y$ such that $C(p,X)$ is homeomorphic to $C(p^{\prime},Y)$ but there is not an homeomorphism between $X$ and $Y$ sending $p$ to $p^{\prime}$.
\end{proposition}

\begin{proof}
With the notation of Example \ref{examples}, we have that $C(p,X)$ is homeomorphic to $C(v,F_{\omega})$ and idem to $C(w,W)$, but $X$ cannot be simultaneously homeomorphic both $F_{\omega}$ and $W$, just choose $Y\in \{F_{\omega},W\}$ such that $Y$ is not homeomorphic to $X$.
\end{proof}

Given a dendrite $X$ and $x,y\in X$ two different points, it is easy to see, by using \cite[Lemma 10.24, p. 175]{Nadler(1992)}, that there exist a unique arc contained in $X$ with end points $x$ and $y$, we will denote this arc by $\overline{xy}$.

Now, let $X\in \mathcal{D}-\mathcal{T}$, if there exists $p\notin sh(X)$, we denote by $sh(X,p)$ the following subset of $sh(X)$, $$sh(p,X)=\{q\in sh(X):\overline{pq}\cap sh(X)=\{q\}\}.$$

\begin{lemma}\label{nonulo}
Let $X\in \mathcal{D}-\mathcal{T}$ and let $p\notin sh(X)$, then:
\begin{enumerate}
\item $sh(p,X)\neq \emptyset$;
\item if $Y\in\mathcal{D}$, $p^{\prime}\in Y$ and there exists $f:X\to Y$ be an homeomorphism sending $p$ to $p^{\prime}$, then $f(sh(p,X))=sh(p^{\prime},Y)$. 
\end{enumerate}
\end{lemma}
\begin{proof}
For the first part, let $r\in sh(X)$ and take $\alpha:[0,1]\to \overline{pr}$ be an homeomorphism such that $\alpha(0)=p$ and $\alpha(1)=r$. Let $t_{0}=\inf\{t\in [0,1]:\alpha(t)\in sh(X)\}$. By using 4. of Proposition \ref{interiorgraph}, we have that $0<t_{0}$, now the continuity of $\alpha$ implies that $\alpha(t_{0})\in sh(X)$. By definition of $t_{0}$ we get that $\alpha(t_{0})\in sh(p,X)$.

To the second part, since the order of points are preserved by homeomorphism it is clear that $f(sh(X))\subset sh(Y)$ and consequently $f(sh(X))=sh(Y)$. The uniqueness of arcs connecting different points in a dendrite implies that the arc connecting $p$ with any other point $q\in sh(p,X)$ is sending by $f$ into the arc connecting $f(p)$ with $f(q)$, and this arc contains no other point of $sh(Y)$ different to $f(q)$. This means that $f(q)\in sh(p^{\prime},Y)$. As before, this shows that $f(sh(p,X))\subset sh(p^{\prime},Y)$ and therefore $f(sh(p,X))=sh(p^{\prime},Y)$.

\end{proof}

Given a dendrite $X$, we say that $A\subset X$ is an \textit{edge in $X$} if there exist an homeomorphism $\alpha:[0,1]\to A$ such that $\alpha(0),\alpha(1)\in E(X)\cup R(X)$ and for each $t\in (0,1)$, $\textrm{ord}(\alpha(t),X)=2$. We say that $A$ is a \textit{free arc of $X$} if the homeomorphism $\alpha$ can be choose in such a way that for each $t\in (0,1)$, $\alpha(t)\in \textrm{Int}_{X}(A)$.  Since each homeomorphism between dendrites preserves the order of its points, we get the following result.

\begin{lemma}\label{free_arcs}
Let $f:X\to Y$ be an homeomorphism between dendrites $X$ and $Y$. Then $A\subset X$ is a free arc of $X$ if and only if $f(A)$ is a free arc in $Y$. Also, if $p\in X$ is an end point of a free arc in $X$ then $f(p)$ is an end point of a free arc in $Y$.
\end{lemma}
In the same setting of Proposition \ref{peludos}, we present the next complementary result.
\begin{proposition}\label{nopeludos}
If $X\in \mathcal{D}-\mathcal{T}$ and $p\notin sh(X)$, then there exists a dendrite $Y$ and $p^{\prime}\in Y$ such that $C(p,X)$ is homeomorphic to $C(p^{\prime},Y)$ but there is not an homeomorphism between $X$ and $Y$ sending $p$ to $p^{\prime}$.
\end{proposition}

\begin{proof}
By Lemma \ref{nonulo} we have that $sh(p,X)\neq \emptyset$. By \cite[10.37, p. 181]{Nadler(1992)} we can 
suppose that $X$ it is contained in the plane $z=0$ of the Euclidean space $\mathbb{R}^{3}$.

We will divide the construction of $Y$ in two cases.

\textit{Case 1.} For each $q\in sh(p,X)$ there is not a free arc of $X$ contained in 
$Cl(X-\overline{pq})$ and with $q$ as an end point. Fix $q\in sh(p,X)$ and define $Y=X\cup (\{q\}\times 
[0,1])$. It is clear that $Y$ is a dendrite containing a free arc in $Cl(X-\overline{pq})$ with $q$ as end point.


\textit{Case 2.} There is $q\in sh(p,X)$ and a free arc $A$ of $X$ contained in $Cl(X-\overline{pq})$ such that $q$ is an end point of $A$. Set $\mathcal{A}$ the family of all free arcs $B$, such that one of its end points, say $q_{B}$, belongs to $sh(p,X)$, and
$B$ is contained in $Cl(X-\overline{pr})$. To construct $Y$, we will attach to each $B\in 
\mathcal{A}$ a countable family of arcs, we do this in the following way. For $B\in \mathcal A$, choose
$\alpha_{B}:[0,1]\to B$ a homeomorphism such that 
$\alpha_{B}(0)=q_{B}$. Now, for each $n\in\mathbb{N}$, let $L^{B}_{n}=\{\alpha_{B}(\frac{1}{n})\}\times 
[0,\frac{1}{n}]$. It is clear that $B\cup \bigcup_{i=1}^{\infty}L^{B}_{n}$ is homeomorphic to the null comb.

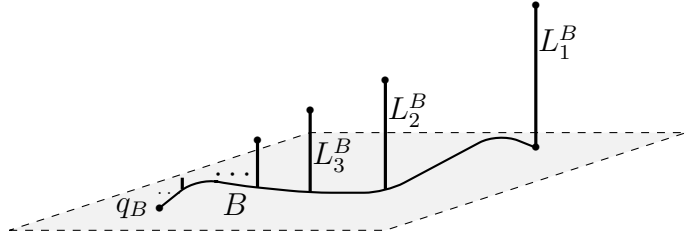
\begin{center}
\begin{tikzpicture}
\draw[dashed, fill=gray!10] (0,0) -- (5,0) -- (9,1.3) -- (4,1.3) -- cycle;
\draw node at (2,0.3){\tiny $\bullet$}
      node at (2,0.3)[left]{$q_B$}
      node at (3,0.35){$B$}
      node at (7,1.1){\tiny $\bullet$};
\draw[rounded corners=8pt, thick] (2,0.3) -- (2.5,0.7) -- (3,0.6) --(4,0.5)--  (5,0.5) -- (6.5,1.3)-- (7,1.1);
  \draw[very thick] (7,1.1) -- (7,3)
		    (5,0.55) -- (5,2)
		    (4,0.5) -- (4,1.6)
		    (3.3,0.57) -- (3.3,1.2)
		    (2.3,0.55) -- (2.3,0.7);
\foreach \x/\y in{7/3,5/2,4/1.6,3.3/1.2}{
  \draw node at (\x,\y){\tiny $\bullet$};
}
\draw node at (3,0.5)[above]{$\cdots$}
      node at (2.15,0.5){\tiny$\cdots$}
      node at (7.3,2.5){$L_1^B$}
      node at (5.3,1.6){$L_2^B$}
      node at (4.3,1){$L_3^B$};
\end{tikzpicture}
\end{center}

We define $Y=X\cup \bigcup_{B\in \mathcal{A}}(B\cup \bigcup^{\infty}_{i=1}L^{B}_{n})$. Since each element in 
$\mathcal{A}$ is a free arc, it is easy to see that $Y$ is a dendrite.

In both cases, we have  construct a denrite $Y$ such that $X\subset Y$. We denote by $p^{\prime}=p\in Y$. 
It is easy to see that $sh(p,X)=sh(p^{\prime},Y)$. Therefore, by Lemmas \ref{nonulo} and \ref{free_arcs}, we 
have that there is not a homeomorphism between $X$ and $Y$, sending $p$ to $p^{\prime}$.

To finish the prove, we need to show that $C(p,X)$ is homeomorphic to $C(p^{\prime},Y)$, but this is in fact 
the case because for each $q\in sh(p,X)=sh(p^{\prime},Y)$, $C(\overline{pq},X)$ and $C(\overline{p 
^{\prime}q},Y)$ are both Hilbert Cubes .

\end{proof}

The following result is a consequence of propositions \ref{peludos} and \ref{nopeludos}.

\begin{corollary}\label{nonuniquedendrites}
Let $X\in\mathcal{D}-\mathcal{T}$ and $p\in X$. There exist $Y\in\mathcal{D}$ and $q\in Y$ such that $C(p,X)$ and $C(q,Y)$ are homeomorphic but $X$ is not homeomorphic to $Y$.
\end{corollary}

The next is the main result of this paper and its proof follows from Lemma \ref{uniquetrees}, Corollary \ref{nonuniquedendrites} and Theorem \ref{trees}. It gives a characterization of dendrites having unique hyperspace $C(p,X)$ in the class $\mathcal{D}$.

\begin{theorem}\label{mainresult}
Let $X\in\mathcal{D}$ and $p\in X$. Then $(X,p)$ has unique hyperspace $C(p,X)$ relative to $\mathcal{D}$ if and only if $X\in \mathcal{T}$.
\end{theorem}

\section{Uniqueness of $C(p,X)$ in the class of continua}

Concerning Question \ref{continua}, in this section we show that for each dendrite or finite graph $X$ and any $p\in X$, $(X,p)$ has not unique hyperspace $C(p,X)$ in the class of continua. Acording Theorem \ref{mainresult}, we have answered the question when $X\in\mathcal{D}-\mathcal{T}$, thus it is sufficient to consider the case of finite graphs.

\begin{theorem}
For each finite graph $X$ and any $p\in X$, there exists a continuum $Y$,  which is not finite graph, and $y\in Y$ such that $C(p,X)$ is homeomorphic to $C(y,Y)$.
\end{theorem}

\begin{proof}
Before we go into the proof, we want to remark a couple of things about to the Knaster continuum, $Z$, with two end points $a$ and $b$ (see \cite[p. 205]{Kuratowski(1968)}). By using \cite[Observation 5.18 and Theorem 5.21]{Pellicer(2003)}, it is easy to see that $Z$ satisfies that $C(a,Z)$ and 
$C(b,Z)$ are arcs and $C(q,Z)$ is a $2$-cell for each $q\in Z-\{a,b\}$.

 \begin{center}
 \begin{tikzpicture}[xscale=6,yscale=5]
 \begin{scope}
 \draw node at (-0.025,0){\tiny$\bullet$}
      node at (-0.025,0)[left]{$a$};
    \draw (0.016,0) -- (0.007,0);
 \clip (-0.1,0) rectangle (0.75,0.5);
  \foreach \r in {0.3,0.292,0.268,0.26,0.22,0.212,0.188,0.18,0.14,0.132,0.108,0.1}{
  \draw ({3/10},0) circle ({7/10 +\r cm});
  }
  \end{scope}
  \begin{scope}
  \clip (0,0) rectangle (0.3,-0.3);
   \foreach \r in{0.06,0.052,0.028,0.02}{
   \draw ({28.85/250},0) circle (\r cm);
   }
  \end{scope}
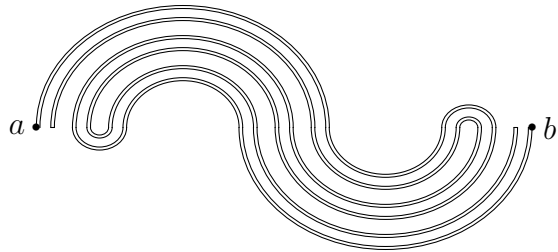
\draw node at (0.5,-0.5){Knaster with two end points};
\begin{scope}[xshift=1.049cm,xscale=-1,yscale=-1]
  \begin{scope}
   \draw node at (-0.025,0){\tiny$\bullet$}
      node at (-0.025,0)[right]{$b$};
 \clip (-0.1,0) rectangle (0.75,0.5);
  \foreach \r in {0.3,0.292,0.268,0.26,0.22,0.212,0.188,0.18,0.14,0.132,0.108,0.1}{
  \draw ({3/10},0) circle ({7/10 +\r cm});
  }
  \end{scope}
  \draw (0.016,0) -- (0.007,0);
    \begin{scope}
  \clip (0,0) rectangle (0.3,-0.3);
   \foreach \r in{0.06,0.052,0.028,0.02}{
   \draw ({28.85/250},0) circle (\r cm);
   }
  \end{scope}
\end{scope}
  \end{tikzpicture}
\end{center}


Given a finite graph $X$ and $p\in X$, suppose that $\mathcal{C}=\{A_{i}:i\in\{1,\cdots ,n\}\}$ is a finite family of arcs such that $X=\cup \mathcal{C}$ and any two elements in $\mathcal{C}$ are either disjoint or intersect only in one or both of their end points. To construct $Y$ we pick one of the arcs in $\mathcal{C}$, say $A_{1}$, such that $p\notin A_{1}$ and we exchange this arc in $X$ for a topological copy of $Z$, say $Z_{1}$, such that the end points of $Z_{1}$ are the same that the end points of $A_{1}$ and for each $i\in\{2,\cdots,n\}$, $Z_{1}\cap A_{i}=A_{1}\cap A_{i}$. 

With this construction, for $y=p\in Y$, $C(p,X)$ and $C(y,Y)$ are homeomorphic but $X$ and $Y$ are not.
\end{proof}

We present now other classes of continua $X$ such that for each $p\in X$, $(X,p)$ has not unique hyperspace $C(p,X)$ in the class of continua.

Recall that, a continuum $X$ is say to be \textit{decomposable} if there exist $A,B\in C(X)-\{X\}$ such that $X=A\cup B$. A continuum $X$ is say to be \textit{indecomposable} if it is not decomposable. Also, $X$ is {\it hereditarily indecomposable (decomposable)} if for each $C\in C(X)$ (and $|C|>1$), $C$ is  indecomposable (decomposable).

Since the Knaster continuum with two end points is indecomposable (see \cite[p. 205]{Kuratowski(1968)}) we have the following result, wich uses a similar construction to the given in the previous theorem.

\begin{theorem}
Let $X$ be a hereditarily decomposable continuum with at least one free arc. For any $p\in X$, there exists a no hereditarily decomposable continuum $Y$ and $q\in Y$ such that $C(p,X)$ and  $C(q,Y)$ are homeomorphic. 
\end{theorem}

The following two results are applications of \cite{Eberhart(1978)} (see the remark about (4) on page 222 and Theorem 8).

\begin{proposition}
Let $X$ is a finite dimensional continuum such that there exists $q\in X$ in such a way that $X$ locally connected in a neighborhood of $q\in X$ and $\dim_{\{q\}}C(q,X)=\infty$. Then for each $p\in X$ there exists a continuum $Y$ and $y\in Y$ such that $C(p,X)$ and $C(y,Y)$ are homeomorphic but $X$ and $Y$ are not.
\end{proposition}

\begin{proposition}\label{dendroid}
Let $X$ is a smooth dendroid at $q\in X$ such that $q$ is not in the interior of a finite tree in $X$. Then for each $p\in X$ there exists a continuum $Y$ and $y\in Y$ such that $C(p,X)$ and $C(y,Y)$ are homeomorphic but $X$ and $Y$ are not $Y$.
\end{proposition}

In both propositions, the construction of $Y$ is done by just attaching a copy of $\mathcal{Q}$, say $K$, at $q\in X$,  in such a way that $X\cap K=\{q\}$. It follows that $C(p,X)=C(p,Y)$. Note that, in the case of Proposition \ref{dendroid}, it is enough to attach a $2$-cell instead of a copy of $\mathcal{Q}$.

In \cite{Pellicer(2003)}, an \textit{end point of a continuum $X$} means a point $p$ such that $C(p,X)$ is an arc (note that this definition of end point is quite different to the one given here). Using this definition of end point, we can obtain other classes of continua without unique hyperspace $C(p,X)$. Thus classes of continua are summarized in the next theorem, which is an easy application of \cite[Example 1.1, p. 2]{Corona.et.al(2019)}, \cite[Theorem 5.22 and Corollary 5.24]{Pellicer(2003)} and the fact that $C(0,[0,1])$ is an arc and $C((1,0),S^{2})$ is a two cell. 

\begin{theorem}
If $X$ belongs to some of one of the following classes of continua: 
\begin{itemize}
\item hereditarily indecomposable,
\item indecomposable without end points and with all proper and nondegenerate subcontinua being arcs,
\item indecomposable with exactly two end points and with all proper and nondegenerate subcontinua being arcs, 
\end{itemize}
then, for each $p\in X$, $(X,p)$ has not unique hyperspace $C(p,X)$.
\end{theorem}

We finish this paper with some open questions.

\begin{question}
If $X$ is a tree and $p\in X$, it is true that $(X,p)$ has unique hyperspace $C(p,X)$ in the class of dendroids?
\end{question}

\begin{question}
If $X$ is a dendroid such that $(X,p)$ has unique hyperspace for some $p\in X$, should $X$ be a tree?
\end{question}

\begin{question}
Is the condition of smoothness necessary in Proposition \ref{dendroid}?
\end{question}

\textbf{Acknowledgments.}  The authors wish to thank Eli Vanney Roblero and Rosemberg Toal\'a for the fruitful discussions. We also want to thank Professors Fernando Mac\'ias Romero and David Herrera Carrasco for asking the questions wich motivated us to write down this paper.


\begin{thebibliography}{00}

\bibitem{CharatonikEberhart(1970)}
J. J. Charatoinik and C. Eberhart, {\it On smooth dendroids}, Fund. Math. {\bf 67} (1970), 297--322. 

\bibitem{Corona.et.al(2019)}
F. Corona--V\'azquez, R. A. Qui\~nones--Estrella, J. S\'anchez--Mart\'inez and R. Toal\'a--Enr\'iquez, {\it Uniqueness of the hyperspaces $C(p,X)$ in the class of trees}, Topology Appl. {\bf 269} (2020) 106926.
\smallskip

\bibitem{Eberhart(1978)}
C. Eberhart, {\it Intervals of continua which are Hilbert Cubes}, Proc. Amer. Math. Soc. {\bf 68} (2) (1978) 220--224.
\smallskip

\bibitem{IllanesNadler(1999)}
A. Illanes and S. B. Nadler, Jr., \textit{Hyperspaces, Fundamentals and recent advances}, Monographs and Textbooks in Pure and Applied Mathematics, {\bf 216}, New York. Marcel Dekker, Inc., 1999. 
\smallskip

\bibitem{Kuratowski(1968)}
K. Kuratowski, Topology, vol. II, Acad. Press, New York, 1968.
\smallskip
	
\bibitem{Veronica(2006)}
V. Mart\'inez de la Vega, {\it Dimension of $n$-fold hyperspaces of graphs}, Houston J. Math. {\bf 32} (3) (2006), 783--799.
\smallskip

\bibitem{Nadler(1978)} 
S. B. Nadler, Jr., {\it Hyperspaces of Sets: A Text with Research Questions}, Monographs and Textbooks in Pure and Applied Mathematics, {\bf 49}, Marcel Dekker, Inc., New York and Basel, 1978.
\smallskip

\bibitem{Nadler(1992)}
S. B. Nadler, Jr., {\it Continuum Theory: An introduction}, Monographs and Textbooks in Pure and Applied Mathematics, {\bf 158}, Marcel Dekker, Inc., New York and Basel, 1992.
\smallskip

\bibitem{Pellicer(2003)}
P. Pellicer, {\it The hyperspaces $C(p,X)$}, Topol. Proc. {\bf 27} (1) (2003) 259--285.

\bibitem{Pellicer(2007)}
P. Pellicer--Covarrubias, {\it Cells in hyperspaces}, Topology Appl. {\bf 154} (2007), 1002--1007.
	
\end{thebibliography}
\end{document}